\documentclass[a4paper,12pt]{amsart}

\usepackage{amssymb,mathrsfs,bm}
\usepackage{mathtools,galois,braket,esint}
\usepackage[utf8]{inputenc}
\usepackage[top=0.8in, bottom=0.8in, left=0.5in, right=0.5in]{geometry}
\usepackage{hyperref,hypcap}
\hypersetup{colorlinks=true,allcolors=black,bookmarksdepth=3}
\usepackage[inline]{enumitem}
\usepackage{array}
\usepackage[utf8]{inputenc}
\usepackage{tikz, tikz-3dplot}
\usepackage{amsmath, amsthm, amssymb,graphics }
\usetikzlibrary{arrows,shapes,positioning}
\usetikzlibrary{intersections}
\usetikzlibrary{decorations.pathreplacing,calligraphy}

\usepackage{amsmath,amsthm,amssymb,amscd,color, xcolor,mathtools,url,tikz}
\usepackage{bbm}
\usepackage{pgfplots}

\theoremstyle{plain}
\newtheorem{theorem}{Theorem}[section]

\theoremstyle{definition}

\theoremstyle{remark}
\newtheorem{remark}[theorem]{Remark}

\numberwithin{equation}{section}
\numberwithin{figure}{section}
\numberwithin{table}{section}
\newcommand{\dd}{\mathop{}\!\mathrm{d}}

\allowdisplaybreaks

\begin{document}

\title{Entropy maximizers for kinetic wave equations set on tori}

\author{Miguel Escobedo} 
\address{Departamento de Matemáticas, Universidad del País Vasco, Apartado 644, E–48080 Bilbao, Spain.}
\email{miguel.escobedo@ehu.es}

\author{Pierre Germain} 
\address{Department of Mathematics, Huxley building, South Kensington campus, Imperial College London, London SW7 2AZ, United Kingdom}
\email{pgermain@ic.ac.uk}

\author{Joonhyun La} 
\address{KIAS School of Mathematics, 85 Hoegi-ro, Dongdaemun-gu, Seoul 02455, Republic of Korea.}
\email{joonhyun@kias.re.kr}

\author{Angeliki Menegaki} 
\address{Department of Mathematics, Huxley building, South Kensington campus, Imperial College London, London SW7 2AZ, United Kingdom}
\email{a.menegaki@imperial.ac.uk}

\maketitle

\begin{abstract}
We consider the kinetic wave equation, or phonon Boltzmann equation, set on the torus (physical system set on the lattice). We describe entropy maximizers for fixed mass and energy; our framework is very general, being valid in any dimension, for any dispersion relation, and even including the quantum kinetic wave equation. Of particular interest is the presence of condensation in certain regimes which we characterize.
\end{abstract}

\section{Introduction}

\subsection{Wave turbulence on lattices}
The central idea in wave turbulence theory is that nonlinear wave equations are amenable to a kinetic description in a regime characterized by weak nonlinearity and turbulence. The application of this idea to Hamiltonian equations set on lattices $\mathbb{Z}^d$ has recently enjoyed renewed interest  to approach questions such as energy transport or heat conduction in anharmonic crystals through oscillator chain models see \cite{fiftyFPU,gallavotti2007fermi, Dhar08, Lep16, LLP03}, for the wave turbulence approach in \cite{LukkarinenSpohn2008,AokiLukkSpohn} and a comprehensive review on the wave turbulence approach, see \cite{ONORATO2023} and references therein. 

Following this idea leads to a nonlinear kinetic equation, known as the kinetic wave equation, or phonon Boltzmann equation, on nonnegative functions defined on the torus\footnote{This means that we are focusing here on the space-homogeneous case, which corresponds to statistical invariance by translation at the level of the Hamiltonian equation} 
$$
\mathbb{T}^d = \mathbb{R}^d / (2\pi \mathbb{Z})^d.
$$
We will not focus on the equation itself here, but rather on its conserved and monotone quantities. The conserved quantities are the mass $\mathcal{M}$ and energy $\mathcal{E}$
$$
\mathcal{M}(f) = \int_{\mathbb{T}^d} f(p) \dd p, \qquad \mathcal{E}(f) = \int_{\mathbb{T}^d} \omega(p) f(p) \dd p.
$$
($f$ is always assumed to be integrable and nonnegative).
Here, $\omega$ is a nonnegative function on the torus given by the dispersion relation for the linear part of the Hamiltonian system at hand - we will come back shortly to its physical origin and give some examples.
The monotone quantity is the entropy, whose definition depends on whether the nonlinear wave equation is classical or quantized, see \cite{Spohn2005}. The classical and quantum entropy are given by
$$
 \mathcal{H}_{cl}(f) = \int_{\mathbb{T}^d} \ln f(p) \dd p, \qquad \mathcal{H}_{qu}(f) = \int_{\mathbb{T}^d} \left[(1+f(p))\ln (1+f(p)) - f(p)\ln f(p) \right] \dd p
$$
respectively.

Since the energy and mass are conserved but the entropy increases, it is natural to expect that attractors of the dynamics will be the \textit{maximizers of the entropy for fixed energy and mass}; the aim of this note is to describe them. These maximizers should be the key to the dynamics of the homogeneous equation but also its hydrodynamic limit, as is the case for the Boltzmann equation.

\subsection{The dispersion relation} \label{sectiondispersion} Neglecting anharmonic terms and demanding translation invariance, we are led to a Hamiltonian of the form
$$
H(p,q) = \frac{1}{2} \sum_{n \in \mathbb{Z}^d} p_n^2 + \frac{1}{2} \sum \alpha(m-n) q_m q_n,
$$
where $p$ and $q$ are real functions on the lattice, and $\alpha$ is even and real-valued. This yields the dynamical equation
$$
\frac{d^2}{dt^2} q_n = - \sum \alpha(m -n) q_m(t),
$$
or, after taking the Fourier transform $\widehat{q}(p) = \sum_{n \in \mathbb{Z}^d} q_n e^{in \cdot p}$, where $p \in \mathbb{T}^d$,
$$
\frac{d^2}{dt^2} \widehat{q}(t,p) = \widehat{\alpha} (p) \widehat{q}(t,p) 
$$
Thus, $\alpha(p) \geq 0$ is a necessary stability condition and the dispersion relation is given by
$$
\omega(p) = \sqrt{\widehat{\alpha (p)}}.
$$

We now review some examples of dispersion relation: nearest neighbor interactions, which are standard, but also long range interactions.
\begin{itemize}
\item Nearest neighbor interaction without pinning corresponds to the discrete Laplacian, or in other words
$$
\alpha(n) = \begin{cases} 2d & \mbox{if $n=0$} \\ -1 & \mbox{if $|n|=1$} \\ 0 & \mbox{otherwise} \end{cases}, \qquad \mbox{giving} \qquad \omega(p) = 2 \sqrt{ \sum_{k=1}^d \sin\left(\frac{p_k}{2} \right)^2}
$$
\item Nearest neighbor interaction with pinning is obtained by adding to $\alpha$ in the previous example $\omega_0^2 \delta_{0,n}$ (Kronecker delta). This gives
$$
\omega(p) = \sqrt{\omega_0^2 + 4  \sum_{k=1}^d \sin\left(\frac{p_k}{2} \right)^2}
$$
{\item For next to nearest neighbor interaction, with or without pinning, the dispersion relation is given by
$$
\sqrt{\omega _0^2+2\sum_{k=1}^d \sin\left( \frac{p_k}2 \right)^4 }
$$
with $\omega _0\ge 0$ (\cite{Spohn2005, Luk2016}).}

\item A typical example of long range interaction is provided by  $\alpha(n) = |n|^{-\delta}$. For simplicity, we assume that there is no pinning $\sum \alpha(n) = 0$ and that the decay condition 
$d < \delta < d +2$ holds (it would not be hard to consider more general situations, but would lead to distinguishing more cases). By Poisson summation, one finds that $\omega(p)$ is smooth except at zero where
$$
\omega(p) \sim |p|^{\frac{\delta -d}{2}}.
$$
\end{itemize}

\underline{We will henceforth assume} the dispersion relation $\omega$ to be continuous and taking its maximal and minimal value on a set of measure zero - both assumptions could be easily relaxed, but they allow for cleaner statements and proofs.

Since $\omega \geq 0$ and by the homogeneity of the problem, we will furthermore assume that 
\begin{equation}
\label{a1}
\min_{\mathbb{T}^d} \omega = a \geq 0, \qquad \max_{\mathbb{T}^d} \omega = 1.
\end{equation}
{
It follows from (\ref{a1}) that for any non negative bounded measure $\lambda$ on  $\mathbb{T}^d$, 
\begin{equation}
\label{admissible}
a\mathcal M(\lambda)\le\mathcal E(\lambda)\le \mathcal M(\lambda)
\end{equation}
and this is then a necessary condition on any pair $\mathcal M, \mathcal E$ to be the mass and energy of some non negative, bounded measure $\lambda$.
}

\subsection{Relaxing the maximization problem} Writing the Euler-Lagrange equation for the constrained maximization of the entropy with fixed mass and energy leads to the Rayleigh-Jeans $R_{\mu,\nu}$ and the Bose-Einstein $B_{\mu,\nu}$ equilibria in the classical and quantum cases respectively
\begin{equation}
\label{RJBE}
R_{\mu,\nu}(p) = \frac{1}{\mu \omega(p) + \nu} , \qquad B_{\mu,\nu}(p) =  \frac{1}{e^{\mu \omega(p) + \nu}-1}.
\end{equation}
We have $R_{\mu,\nu},B_{\mu,\nu} \geq 0$ if and only if
\begin{equation}
\label{rangemunu}
\mbox{either} \;\; \begin{cases} \nu \geq 0 \\ \mu \geq - \nu \end{cases} \qquad \mbox{or}  \;\; \begin{cases} \nu \leq 0 \\ \mu \geq - \frac \nu a \end{cases}.
\end{equation}

However, it is shown below that, for certain choices of $\omega $, there are values of $\mathcal M$ and $\mathcal E$, satisfying (\ref{admissible}) for which there is no entropy maximizers of the form (\ref{RJBE}).
 
As it was observed first in \cite{Einstein} for ideal Bose gases, 
the solution to this difficulty is to relax the problem at hand by allowing general non negative measures as possible maximizers (cf. for example  \cite{EscMiscValle} for a detailed description). This leads to maximizers with singular parts, similar to the  Bose-Einstein distributions in presence of a condensate, for the Nordheim equation.

In order to set up our minimization problem rigorously for general measures, we decompose a nonnegative measure $\lambda$ by the Radon-Nikodym theorem into a part which is absolutely continuous with respect to the Lebesgue measure, and a part which is singular:
\begin{equation}
\label{RadonNikodym}
\dd \lambda = f \dd p + \dd \lambda_{\operatorname{sing}}.
\end{equation}
Following \cite{DemengelTemam,EscMiscValle}, the definitions of the mass, energy and entropy can be extended to general measures as follows
\begin{align*}
\begin{cases}
& \displaystyle \mathcal{M}(\lambda) = \int_{\mathbb{T}^d} \dd \lambda \\ 
& \displaystyle \mathcal{E}(\lambda) = \int_{\mathbb{T}^d} \omega \dd \lambda
\end{cases}
\qquad 
\begin{cases}
& \displaystyle  \mathcal{H}_{cl}(\lambda) = \int_{\mathbb{T}^d} \ln f(p) \dd p \\
& \displaystyle \mathcal{H}_{qu}(\lambda) = \int_{\mathbb{T}^d} \left[(1+f(p))\ln (1+f(p)) - f(p)\ln f(p) \right] \dd p
\end{cases}
\end{align*}
The definitions for the mass and energy are natural; as for the entropy of a general measure, we define it to be equal to the absolutely continuous part of the measure. In heuristic terms, the logarithmic growth of the entropy functional cancels the contribution of the singular part of the measure.

\subsection{Main result} 
Note that
\begin{equation}
\label{EMa1}
\mbox{for any $\lambda$,} \qquad a \leq \frac{\mathcal{E}(\lambda)}{\mathcal{M}(\lambda)} \leq 1.
\end{equation}
Defining
\begin{equation}
\label{defalphabeta}
\alpha = a + (2\pi)^d \left( \int_{\mathbb{T}^d}  \frac{\dd p}{\omega - a} \right)^{-1}, \qquad \beta =  1 - (2\pi)^d \left( \int_{\mathbb{T}^d}  \frac{\dd p}{1- \omega} \right)^{-1},
\end{equation}
we see that
$$
a \leq \alpha < \beta \leq 1.
$$
Furthermore, the first and second inequalities are equalities if and only if $\int \frac{\dd p}{\omega - a} = \infty$ and $\int \frac{\dd p}{1- \omega} = \infty$ respectively.

\begin{theorem} [Maximizers of the classical entropy]
\label{maintheorem}
We want to describe the maximizers of the entropy $\mathcal{H}_{cl}$ over positive measures subject to the constraints $\mathcal{M}(f) = \mathcal{M}_0$ and $\mathcal{E}(f) = \mathcal{E}_0$.
\begin{itemize}
\item[(i)] If $\displaystyle \alpha < \frac{\mathcal{E}_0}{\mathcal{M}_0} < \beta$, the unique maximizer is the unique RJ equilibrium with this mass and energy.
\item[(ii)] If $\displaystyle \beta \leq \frac{\mathcal{E}_0}{\mathcal{M}_0} < 1$, maximizers are of the type $R_{\mu,\nu} + \lambda_{\operatorname{sing}}$, where $(\mu,\nu)$ are characterized by 
$$
\begin{cases}
\mathcal{M}(R_{\mu,\nu}) = \frac{\mathcal{M}_0 - \mathcal{E}_0}{1-\beta} \\ \mathcal{E}(R_{\mu,\nu}) = \beta \frac{\mathcal{M}_0 - \mathcal{E}_0}{1-\beta},
\end{cases}
$$
and $\lambda_{\operatorname{sing}}$ has mass $\frac{1}{1-\beta}(\mathcal{E}_0 - \beta \mathcal{M}_0)$ and is supported on the set where $\omega$ is maximal.

\item[(iii)] If $\displaystyle a < \frac{\mathcal{E}_0}{\mathcal{M}_0} \leq \alpha$, maximizers are of the type $R_{\mu,\nu} + \lambda_{\operatorname{sing}}$ with 
$$
\begin{cases}
\mathcal{M}(R_{\mu,\nu}) = \frac{\mathcal{M}_0 - \mathcal{E}_0}{1-\alpha} \\ \mathcal{E}(R_{\mu,\nu}) = \alpha \frac{\mathcal{M}_0 - \mathcal{E}_0}{1-\alpha}.
\end{cases}
$$ Also $\lambda_{\operatorname{sing}}$ has mass and energy equal to $\frac{1}{1-\alpha}(\mathcal{E}_0 - \alpha \mathcal{M}_0)$ and is supported on the set where $\omega$ is maximal.
\end{itemize}
\end{theorem}

\begin{theorem}[Maximizers of the quantum entropy]
    \label{Quantum_maintheorem}
Let $\mathcal{M}_0$ and $\mathcal{E}_0$ $\in (0,\infty)$ be a given mass and energy. We describe the maximizers of the entropy $\mathcal{H}_{qu}$ under the constraints that $\mathcal{M}(\lambda) = \mathcal{M}_0$ and $\mathcal{E}(\lambda) = \mathcal{E}_0$. Let $f_{\pm}$ be two differentiable, increasing functions satisfying $f_+(0) = f_-(0)=0$,  with the following linear asymptotic behavior as $\mathcal{M} \to \infty$:
$$ 
f_+(\mathcal{M}) \approx \beta \mathcal{M} \text{ and }
f_-(\mathcal{M}) \approx \alpha \mathcal{M}, 
$$
where $\alpha$ and $\beta$ are defined in \eqref{defalphabeta}. 
We split into the following cases
\begin{itemize}
    \item[(i)] If $f_- (\mathcal{M}_0) < \mathcal{E}_0 < f_+ (\mathcal{M}_0) $, the unique maximizer is the unique Bose-Einstein equilibrium distribution with this mass and energy. This corresponds to the interior of region $S$ in Figure \ref{fig:Quantum picture}. 
    \item[(ii)] If $f_+ (\mathcal{M}_0) \leq \mathcal{E}_0 < \mathcal{M}_0$, the maximizers have the form $R_{\mu,\nu} +\lambda_{\operatorname{sing}}$ and $\lambda_{sing}$ are concentrated in the regions where $\omega=1$. 
 This corresponds to the gray region above the cone in Figure \ref{fig:Quantum picture}.
    \item[(iii)] If $ a \mathcal{M}_0 < \mathcal{E}_0 \leq f_- (\mathcal{M}_0)$, an analogous statement as in (ii) holds with the maximizers being $R_{\mu,\nu} +\lambda_{\operatorname{sing}}$ and $\lambda_{sing}$. This corresponds to the gray region below the cone in Figure \ref{fig:Quantum picture}.
\end{itemize}
\end{theorem}

\begin{remark}
If the maximizer is not an integrable function but rather comprises a pure point measure, we say that condensation occurs.
The above theorems give a characterization of dispersion relations for which this is the case: it suffices to check whether $\int \frac{dp}{\omega(p) - \min \omega}$ and $\int \frac{dp}{ \max \omega - \omega(p)}$ are finite or not. Applying this criterion to the examples of Section \ref{sectiondispersion}, we see that condensation occurs 
\begin{itemize}
\item For nearest neighbor interaction without pinning, if $d \geq 2$.
\item For nearest neighbor interaction with pinning, if $d \geq 3$.
\item For long range interactions without pinning, if $d < \delta < d+2$.
\item For the next-to-nearest neighbor interaction without pinning, if $d\geq 3$.
\end{itemize}
\end{remark}

\subsection{Discussion}
In \cite{Spohn2005,Spohn2006}, the collisional invariants of the $4$-phonon Boltzmann equation were characterized for a large class of dispersion relations on $\mathbb{R}^d$ for $d\geq 2$. Namely, under the assumption that $\omega$ is smooth with bounded Hessian outside a manifold of codimension at least $1$, it was shown that the integrable collisional invariants are of the form $\mu \omega(p) + \nu$, where $\mu, \nu$ are 
constants. Then the Rayleigh-Jeans distribution arises as a natural candidate for equilibrium because it is a solution to the Boltzmann equation.

Here we rather characterize the equilibrium distribution constrained not only by the collisional invariants but also by the additional requirements on the initial mass and energy, which are the conserved quantities. We conclude that these additional constraints might require, depending on the specific values of mass and energy, a singular part in the distribution. 
Our results reflect the interplay between the dynamics (as captured by the collisional invariants in Spohn's work) and the thermodynamic constraints imposed on the equilibrium state.


For energies and masses such that the maximizer of the entropy is singular, it is natural to ask whether such singularities can manifest dynamically in the Boltzmann equation. For initial conditions aligned with such vales of energy and mass, non regular behaviors of the system could be expected. It is not clear for example if global solutions with such initial data could still converge to these singular maximizers of the entropy, 
or if the ergodicity would break down, in which case the entropy maximizer would not necessarily correspond to the system’s attractor. In view of previous examples (like the wave kinetic equation associated with the nonlinear Schrödinger equation, or the Nordheim equation for a dilute homogeneous gas of bosons, \cite{Esc2015}), finite time blow up and singularity formation could also be expected for suitable dispersion relations. And, in the context of microscopic FPUT chains, could this correspond to a localization of energy in certain modes, (e.g. breathers), the emergence of coherent structures? or an ergodicity breakdown?

\section{The classical case}

\subsection{Rayleigh-Jeans equilibria with prescribed mass and energy}

Recall that
$$
\alpha = a + (2\pi)^d \left( \int \frac{\dd p}{\omega - a} \right)^{-1} > a, \qquad \beta =  1 - (2\pi)^d \left( \int \frac{\dd p}{1- \omega} \right)^{-1} < 1.
$$

To know whether there exists a RJ equilibrium with given mass and energy, it suffices to compare the quotient of the mass and energy to the thresholds $\alpha$ and $\beta$ - this is the content of the following theorem.

\begin{theorem} \label{Theorem1}
Given $\mathcal{M}_0, \mathcal{E}_0 \in (0,\infty)$, we want to understand whether a Rayleigh-Jeans (RJ) equilibrium $R_{\mu,\nu}$ exists, which satisfies
\begin{equation}
\label{M0E0}
\mathcal{M}(R_{\mu,\nu}) = \mathcal{M}_0, \qquad \mathcal{E}(R_{\mu,\nu}) = \mathcal{E}_0.
\end{equation}
\begin{itemize}
\item[(i)] If $ \displaystyle
\alpha < \frac{\mathcal{E}_0}{\mathcal{M}_0} < \beta$, then for any $\mathcal{M}_0, \mathcal{E}_0 > 0$, there exists a unique RJ equilibrium satisfying \eqref{M0E0}.
\item[(ii)] If $\displaystyle \frac{\mathcal{E}_0}{\mathcal{M}_0} = \beta <1$ (which implies that $\displaystyle \int \frac{\dd p}{1- \omega} < \infty$), there exists a unique RJ equilibrium satisfying \eqref{M0E0}. It is furthermore such that $\mu = -\nu$, $\nu > 0$.
\item[(iii)] If $\displaystyle \frac{\mathcal{E}_0}{\mathcal{M}_0} = \alpha >a$ (which implies that $\displaystyle \int \frac{\dd p}{\omega-a} < \infty$), there exists a unique RJ equilibrium satisfying \eqref{M0E0}. It is furthermore such that $\nu = -\mu a$, $\mu > 0$.
\item[(iv)] If $\displaystyle \frac{\mathcal{E}_0}{\mathcal{M}_0} > \beta$ or $\displaystyle \frac{\mathcal{E}_0}{\mathcal{M}_0} < \alpha$, there does not exist a RJ equilibrium satisfying \eqref{M0E0}.
\end{itemize}
\end{theorem}

\begin{proof} 
We will denote $\mathcal{M}(\mu,\nu) = \mathcal{M}(R_{\mu,\nu})$. It follows from the identity
\begin{equation}
\label{identitymunu}
\mu \mathcal{E}(\mu,\nu) + \nu \mathcal{M}(\mu,\nu) = (2\pi)^d
\end{equation}
that
$$
\mathcal{E}(\mu,\nu) = \frac{1}{\mu}((2\pi)^d - \nu \mathcal{M}(\mu,\nu))
$$
Parametrizing $\mu$ and $\nu$ as
$$
\mu = \rho \cos \varphi, \qquad \nu = \rho \sin \varphi,
$$
there holds
$$
\mathcal{M}(\mu,\nu)) = \frac{1}{\rho} \mathcal{M}(\cos \varphi,\sin \varphi) = \frac{1}{\rho} \mathcal{M}(\varphi), \qquad \mathcal{E}(\mu,\nu)) = \frac{1}{\rho} \mathcal{E}(\cos \varphi,\sin \varphi) = \frac{1}{\rho \cos \varphi}((2\pi)^d - \sin \varphi \mathcal{M}(\varphi)).
$$
One can then eliminate $\rho$ and obtain that the equation \eqref{M0E0} is equivalent to
$$
\frac{\mathcal{E}_0}{\mathcal{M}_0} = \frac{(2\pi)^d}{\cos \varphi \mathcal{M} (\varphi)} - \tan \varphi,
$$
or in other words
$$
\frac{\mathcal{E}_0}{\mathcal{M}_0} = F(\tan \varphi) \quad \mbox{with} \quad F(x) = (2\pi)^d \left( \int \frac{\dd p}{\omega + x} \right)^{-1} - x.
$$
The allowed range for $\mu$ and $\nu$ in \eqref{rangemunu} translates into the restriction that 
$$
\varphi \in (\varphi^*,\frac{3\pi}{4}) \qquad \mbox{with} \;\;\varphi^* = - \arctan (a).
$$
which means that 
$$
\tan \varphi \in (-\infty,-1) \cup (-a,\infty).
$$
There remains to compute the image of that set by $F$!

First, we notice that $F$ is strictly increasing since
$$
F'(x) = \frac{\frac{1}{(2\pi)^d} \int \frac{\dd p}{(\omega+x)^2}}{\left( \frac{1}{(2\pi)^d} \int \frac{\dd p}{\omega+x} \right)^2} - 1 > 0
$$
by Jensen's inequality.

Thus, there remains to compute the value of $F$ at the points $-\infty, -1, -a, \infty$.
\begin{itemize}
\item $\displaystyle F(-a+)= a + (2\pi)^d \left( \int \frac{1}{\omega -a} \dd p \right)^{-1}$. If $\int \frac{1}{\omega -a} < \infty$, the limiting RJ equilibrium (corresponding to $\varphi \to \varphi^*$) has finite mass. It is given by $\varphi = \varphi^*$, or in other words $\nu = -a \mu$, $\mu>0$. 
\item $\displaystyle F(-1-) = 1 - (2\pi)^d \left( \int \frac{1}{1 - \omega} \dd p \right)^{-1}$. If $\int \frac{1}{1- \omega} < \infty$, the limiting RJ equilibrium (corresponding to $\varphi \to \frac {3\pi} 4$) has finite mass. It is given by $\varphi = \frac{3\pi}4$ or in other words $\mu = -\nu$, $\nu>0$.
\item $\displaystyle F(\infty) = \frac{I}{(2\pi)^d}$, with $\displaystyle I = \int \omega \dd p$. Indeed, we can expand
$$
\frac{1}{\omega + x} = \frac{1}{x} - \frac{\omega}{x^2} + O_{\infty}(\frac{1}{x^3}) \qquad \mbox{so that} \qquad \int \frac{\dd p}{\omega + x} \sim \frac{(2\pi)^d}{x} - \frac{I}{x^2} + O_\infty \left( \frac 1 {x^3} \right).
$$
We could have found this expression directly since the limit $\tan \varphi \to \infty$ corresponds to the limit $\varphi \to \frac \pi 2 -$ or equivalently $\mu \to 0+$, in which case $\frac{\mathcal{E}(R_{0,\nu})}{\mathcal{M}(R_{0,\nu})} = \frac{I}{(2\pi)^d}$.
\item $\displaystyle F(-\infty) = \frac{I}{(2\pi)^d}$ since the expansion above remains valid as $x \to -\infty$. Here again, the limit $\tan \varphi \to - \infty$ corresponds to the limit $\varphi \to\frac \pi 2 +$ or equivalently $\mu \to 0-$, in which case $\frac{\mathcal{E}(R_{0,\nu})}{\mathcal{M}(R_{0,\nu})} = \frac{I}{(2\pi)^d}$.
\end{itemize}
Therefore, we find that the image of $(-\infty,-1) \cup (-a,\infty)$ by $F$ is $(a,1)$, which was the desired result. 
\end{proof}

\subsection{Proof of Theorem \ref{maintheorem}}

\begin{proof} We start with a measure $\lambda$ satisfying
$$
\mathcal{M}(\lambda) = \mathcal{M}_0, \qquad \mathcal{E}(\lambda) = \mathcal{E}_0
$$
which we decompose as in \eqref{RadonNikodym}. Applying the inequality $\ln x \leq x -1$ to $\frac{f}{R_{\mu,\nu}}$ and integrating gives
$$
\int \ln f \dd p - \int \ln R_{\mu,\nu} \dd p \leq \int f (\mu \omega + \nu) \dd p - (2\pi)^d.
$$
Starting from this inequality and using in addition that $\mu \omega + \nu \geq 0$ and the identity \eqref{identitymunu}, we obtain
\begin{align*}
\mathcal{H}_{cl}(\lambda) = \mathcal{H}_{cl}(f) & \leq \mathcal{H}_{cl}(R_{\mu,\nu}) + \int f (\mu \omega + \nu) \dd p - (2\pi)^d \\
& \leq \mathcal{H}_{cl}(R_{\mu,\nu}) + \int (\mu \omega + \nu)(f \dd p + \dd \lambda_{\operatorname{sing}}) - \mu \mathcal{E}(R_{\mu,\nu}) - \nu \mathcal{M}(R_{\mu,\nu}).
\end{align*}
By the mass and energy constraints on $\lambda$, this means that
\begin{equation}
\label{inegalitedebase}
\mathcal{H}_{cl}(\lambda) \leq \mathcal{H}_{cl}(R_{\mu,\nu}) + \mu (\mathcal{E}_0-\mathcal{E}(R_{\mu,\nu})) + \nu(\mathcal{M}_0 -\mathcal{M}(R_{\mu,\nu})).
\end{equation}

\medskip
\noindent $(i)$ 
Choosing $\mu$ and $\nu$ in the above equation such that $\mathcal{E}(R_{\mu,\nu}) = \mathcal{E}_0$ and $\mathcal{M}(R_{\mu,\nu}) = \mathcal{M}_0$ (which is possible by Theorem \ref{Theorem1} since $\alpha < \frac{\mathcal{E}_0}{\mathcal{M}_0} < \beta$) gives the desired result.

\medskip
\noindent $(ii)$ We choose $\mu$ and $\nu$ such that
$$
\begin{cases}
\mathcal{E}_0 - \mathcal{E}(R_{\mu,\nu}) = \mathcal{M}_0 - \mathcal{M}(R_{\mu,\nu}) \\
\mathcal{E}(R_{\mu,\nu}) = \beta \mathcal{M}(R_{\mu,\nu})
\end{cases}
\quad 
\mbox{or equivalently}
\quad
\begin{cases}
\mathcal{M}(R_{\mu,\nu}) = \frac{\mathcal{M}_0 - \mathcal{E}_0}{1-\beta} \\ \mathcal{E}(R_{\mu,\nu}) = \beta \frac{\mathcal{M}_0 - \mathcal{E}_0}{1-\beta}.
\end{cases}
$$
Such $\mu$ and $\nu$ exist by Theorem \ref{Theorem1}, which gives furthermore that $\mu = -\nu$, with $\nu>0$. With this choice of $\mu$ and $\nu$, inequality \eqref{inegalitedebase} becomes $\mathcal{H}_{cl}(\lambda) \leq \mathcal{H}_{cl}(R_{\mu,\nu})$.

Examining the equality case in the derivation of inequality \eqref{inegalitedebase}, we see that necessarily $f = R_{\mu,\nu}$. This implies that $\lambda_{\operatorname{sing}}$ is such that
$$
\begin{cases}
\mathcal{M}(\lambda_{\operatorname{sing}}) = \mathcal{M}_0 - \mathcal{M}(R_{\mu,\nu}) = \mathcal{M}_0 - \frac{\mathcal{M}_0 - \mathcal{E}_0}{1-\beta} = \frac{1}{1-\beta}(\mathcal{E}_0 - \beta \mathcal{M}_0) \\
\mathcal{E}(\lambda_{\operatorname{sing}}) = \mathcal{E}_0 - \mathcal{E}(R_{\mu,\nu}) = \mathcal{E}_0 -  \beta \frac{\mathcal{M}_0 - \mathcal{E}_0}{1-\beta} = \frac{1}{1-\beta}(\mathcal{E}_0 - \beta \mathcal{M}_0).
\end{cases}
$$
Since the mass and energy of $\lambda_{\operatorname{sing}}$ are equal, it is necessarily supported on the set where $\omega$ takes the value $1$. Conversely, such a measure $\mu$ achieves the maximum value for the entropy, namely $\mathcal{H}_{cl}(R_{\mu,\nu})$.
\end{proof}

\section{The quantum Case} 

\subsection{Bose-Einstein equilibria with prescribed mass and energy}

\begin{theorem} \label{theorem1_Quantum} Given $\mathcal{M}_0,\mathcal{E}_0$, we want to understand whether a Bose-Einstein (BE) equilibrium $B_{\mu,\nu}$ exists which satisfies
\begin{equation}
\label{REM}
\mathcal{M}(B_{\mu,\nu}) = \mathcal{M}_0, \quad \mathcal{E}(B_{\mu,\nu}) = \mathcal{E}_0.
\end{equation}
\begin{itemize}
\item[(i)] If $\displaystyle \int \frac{\dd p}{\omega -a} = \int \frac{\dd p}{1- \omega} = \infty$, then for any $\mathcal{E}_0$ and $\mathcal{M}_0$, there exists a unique (BE) equilibrium satisfying \eqref{REM}.
\item[(ii)] If $\displaystyle \int \frac{\dd p}{\omega -a} + \int \frac{\dd p}{1- \omega} < \infty$, define the two parameterized curves
in $(\mathcal{M},\mathcal{E})$ space:
\begin{align*}
& \mathcal{C}_+ = \{ (\mathcal{M}(B_{\mu,\nu}), \mathcal{E}(B_{\mu,\nu})), \; (\mu,\nu) = (-t,t), \; t \in (0,\infty) \} \\
& \mathcal{C}_- = \{ (\mathcal{M}(B_{\mu,\nu}), \mathcal{E}(B_{\mu,\nu})), \; (\mu,\nu) = (t,-at), \; t \in (0,\infty) \}.
\end{align*}
Both curves can be represented by a graph over the mass $\mathcal{M}$:
$$
\mathcal{C}_\pm = \{ (\mathcal{M}, f_\pm(\mathcal{M}), \; \mathcal{M} \in (0,\infty).\}
$$
The functions $f_{\pm}$ are differentiable, increasing, 
$$
f_+(0) = f_-(0) \qquad 
\mbox{and as $\mathcal{M} \to \infty$}, \qquad 
\begin{cases}
f_+(\mathcal{M}) \approx \alpha \mathcal{M} \\
f_-(\mathcal{M}) \approx \beta \mathcal{M}
\end{cases}
$$
(recall that $\alpha$ and $\beta$ are defined in \eqref{defalphabeta}).

Let $S$ be the subset of $(\mathcal{M},\mathcal{E})$ bounded above by $\mathcal{C}_+$ and below by $\mathcal{C}_-$. Then there exists a (BE) equilibrium satisfying \eqref{REM} if and only if $(\mathcal{M}_0,\mathcal{E}_0) \in S$, and then it is unique.
\end{itemize}
\end{theorem}

For the sake of concision, the statement of the theorem only addresses the cases where both of $\int \frac{\dd p}{\omega -a}$ and $\int \frac{\dd p}{1- \omega}$ are finite or infinite, but the extension to the remaining cases is obvious.

\begin{proof} $(i)$ is the simpler case, thus it will be omitted and we focus on the second case.

\medskip

\noindent
$(ii)$ We start by proving the desired properties of $\mathcal{C}_+$ and $\mathcal{C}_-$. First observe that the finiteness of $\mathcal{M}(B_{-t,t})$ for some $t$ is equivalent to the finiteness for any $t$, and also equivalent to the finiteness of $\int \frac{\dd p}{1-\omega}$. Thus, the curves $\mathcal{C}_+$ and $\mathcal{C}_-$ are well-defined. 

Furthermore, 
\begin{align*}
& \frac{d}{dt} \mathcal{M}(B_{-t,t}) = \frac{d}{dt}  \int \frac{\dd p}{e^{t-t\omega}-1} = - \int \frac{1-\omega}{(e^{t-t\omega}-1)^2} \dd p < 0 \\
& \frac{d}{dt} \mathcal{E}(B_{-t,t}) = \frac{d}{dt}  \int \frac{\omega \dd p}{e^{t-t\omega}-1} =- \int \frac{\omega(1-\omega)}{(e^{t-t\omega}-1)^2} \dd p < 0
\end{align*}
(also notice that the integrals above are finite if $\int \frac{\dd p}{1-\omega} < \infty$), so that $f_{+}$ is increasing, with a similar argument for $f_{-}$. 

On the one hand, if $|\mu|, |\nu| \to \infty$, then $\mathcal{M}(B_{-t,t}) + \mathcal{E}(B_{-t,t}) \to 0$, which implies that $f_+(0) = f_-(0) = 0$. On the other hand, if $\mu,\nu \to 0$, then $\mathcal{M}(B_{\mu,\nu}) \approx \mathcal{M}(R_{\mu,\nu})$ and $\mathcal{E}(B_{\mu,\nu}) \approx \mathcal{E}(R_{\mu,\nu})$, which gives the linear equivalents for $f_{\pm}$ at $\infty$.

We now consider the function
$$
F(\rho,\varphi) = (\mathcal{M}(B_{\mu,\nu}), \mathcal{E}(B_{\mu,\nu})), \quad \mbox{with} \quad \begin{cases} \mu = \rho \cos \varphi \\ \nu = \rho \sin \varphi \end{cases}
$$
which, due to the restriction \eqref{rangemunu}, is defined on the domain
$$
D = \{ (\rho,\phi), \;\rho >0, \; \varphi \in [- \arctan a , \frac{3\pi}{4}] \}.
$$
We claim that $F$ is bijection from $D$ to $S$, which is equivalent to the desired statement. Our argument is as follows: consider the curves
$$
\mathcal{C}_\varphi = \{ F(\rho,\varphi), \; \rho>0 \}.
$$
These curves originate in $(0,0)$ and go to $\infty$; obviously $\mathcal{C}_{\frac{3\pi}{4}} = \mathcal{C}_{+}$ and $\mathcal{C}_{- \arctan a} = \mathcal{C}_{-}$.  We claim that they are ordered ($\mathcal{C}_\phi$ above $\mathcal{C}_\varphi'$ for $\varphi > \varphi'$ and foliate $S$. To check that this is the case, we compute that
\begin{align*}
\begin{cases}
\partial_\rho \mathcal{M}(B_{\mu,\nu}) = - \sin \varphi A - \cos \varphi B \\
\partial_\rho \mathcal{E}(B_{\mu,\nu}) =  - \sin \varphi B - \cos \varphi C,
\end{cases}
\qquad
\begin{cases}
\partial_\varphi \mathcal{M}(B_{\mu,\nu}) = \rho \sin \varphi B - \rho \cos \varphi A \\
\partial_\varphi \mathcal{E}(B_{\mu,\nu}) =  \rho \sin \varphi C - \rho \cos \varphi B,
\end{cases}
\end{align*}
where
$$
A = \int \frac{e^{\mu \omega + \nu}}{(e^{\mu \omega + \nu} -1)^2} \dd p, \qquad B = \int \frac{e^{\mu \omega + \nu} \omega}{(e^{\mu \omega + \nu} -1)^2} \dd p, \qquad C = \int \frac{e^{\mu \omega + \nu} \omega^2}{(e^{\mu \omega + \nu} -1)^2} \dd p.
$$

We learn two things from these formulas. First, 
since $\partial_\rho \mathcal{M}(B_{\mu,\nu})<0$ and $\partial_\rho \mathcal{E}(B_{\mu,\nu})<0)$, we deduce 
that each $\mathcal{C}_\varphi$ is the graph of $\mathcal{E}$ as an increasing function of $\mathcal{M}$.

Second, we can compute the exterior product
$$
\partial_\rho F(\rho,\varphi) \times \partial_\varphi F(\rho,\varphi) = AC - B^2 > 0,
$$
as follows by the Cauchy-Schwarz inequality. It has a constant sign, which means that $\partial_\varphi F$ is always pointing in the same direction compared to $\partial_\rho F$, and thus that the curves are ordered.

Finally, in the limit $(\mu,\nu) \to 0$ or $(\mathcal{E},\mathcal{M}) \to \infty$, everything converges to the classical case, which was the object of the previous section.

Putting together the arguments above, we obtain the desired result, namely that the points in $S$ are exactly the mass and energy of Bose-Einstein condensates, and that there is a one-to-one correspondence between points in $S$ and Bose-Einstein condensates.
\end{proof}

\subsection{Proof of Theorem \ref{Quantum_maintheorem}}

\begin{figure}
\label{figureEM}
\begin{tikzpicture}[>=stealth,scale=1.2]

\def\a{0.1}   
\def\b{0.8}   
\def\alph{0.4} 

\draw[->] (-1,0) -- (10,0) node[right] {$\mathcal{M}$};
\draw[->] (0,-1) -- (0,8.5) node[above] {$\mathcal{E}$};

\draw[thick,black] (0,0) -- (8,8) node[anchor=west] {$\mathcal{E} = \mathcal{M}$};
\draw[thick,black] (0,0) -- (9,{9*\a}) node[anchor=west] {$\mathcal{E} = a\mathcal{M}$};

\begin{scope}
    \clip (0,0) -- plot[smooth,domain=0:10,samples=200] (\x,{(\b*\x)-0.1*sin(deg(3*\x))/sqrt(1+\x)})
          -- plot[smooth,domain=10:0,samples=200] (\x,{(\alph*\x)+0.1*sin(deg(3*\x))/sqrt(1+\x)}) -- cycle;
    \fill[blue!6] (0,0) rectangle (10,8.5);
\end{scope}
\begin{scope}
    \clip (0,0) -- plot[smooth,domain=0:10,samples=200] (\x,{(\b*\x)-0.1*sin(deg(3*\x))/sqrt(1+\x)})
          -- plot[smooth,domain=10:0,samples=200] (\x,{(\x)}) -- cycle;
  \fill[gray!3, dashed] (0,0) rectangle (8,8.5);
\end{scope}
\begin{scope}
    \clip (0,0) -- plot[smooth,domain=0:10,samples=200] (\x,{(\alph*\x)+0.1*sin(deg(3*\x))/sqrt(1+\x)})
         -- plot[smooth,domain=10:0,samples=200] (\x,{(\a*\x)}) -- cycle;
    \fill[gray!3, dashed] (0,0) rectangle (9,8.5);
\end{scope}
\node at (5,3.3) [anchor=north] {$S$};

\draw[thick,black,smooth,domain=0:10,samples=200] 
    plot (\x,{(\b*\x)-0.1*sin(deg(3*\x))/sqrt(1+\x)}) node[anchor=south west] {{\textbf{$\mathcal{C}_+$}}};
\draw[thick,black,smooth,domain=0:10,samples=200] 
    plot (\x,{(\alph*\x)+0.1*sin(deg(3*\x))/sqrt(1+\x)}) node[anchor=north] {$\mathcal{C}_-$};

\draw[dashed,black] (5,{5*\b-0.1}) -- (11,{11*\b-0.1}) node[anchor=west] {$\mathcal{E} = \beta \mathcal{M}$};
\draw[dashed,black] (5,{5*\alph+0.1}) -- (11,{11*\alph+0.1}) node[anchor=west] {$\mathcal{E} = \alpha \mathcal{M}$};

\draw[] (5.5,{5.5*\alph+0.1}) -- (7.5,{7.5*\alph+0.1}) node[midway,below,sloped,rotate=0] 
    {\scriptsize $(\mathcal{M}(B_{\mu,\nu}), \mathcal{E}(B_{\mu,\nu}); \mu = -\nu)$};
    
\draw[] (5.5,{5.5*\b-0.1}) -- (6.5,{6.5*\b-0.1}) node[midway,above,sloped,rotate=0] 
        {\scriptsize $(\mathcal{M}(B_{\mu,\nu}), \mathcal{E}(B_{\mu,\nu}); \nu = -a \mu)$};

\end{tikzpicture}
\caption{For a given mass and energy, $(\mathcal{M}_0, \mathcal{E}_0) \in S$, the quantum entropy maximizer is a regular Bose distribution $B_{\mu,\nu}$. In the remaining light gray regions, an additional singular measure is needed to maximize the entropy. The dashed lines represent the asymptotic with $\mathcal{M}$ behavior of the curves $\mathcal{C}_{\pm}$.}
 \label{fig:Quantum picture}
\end{figure}
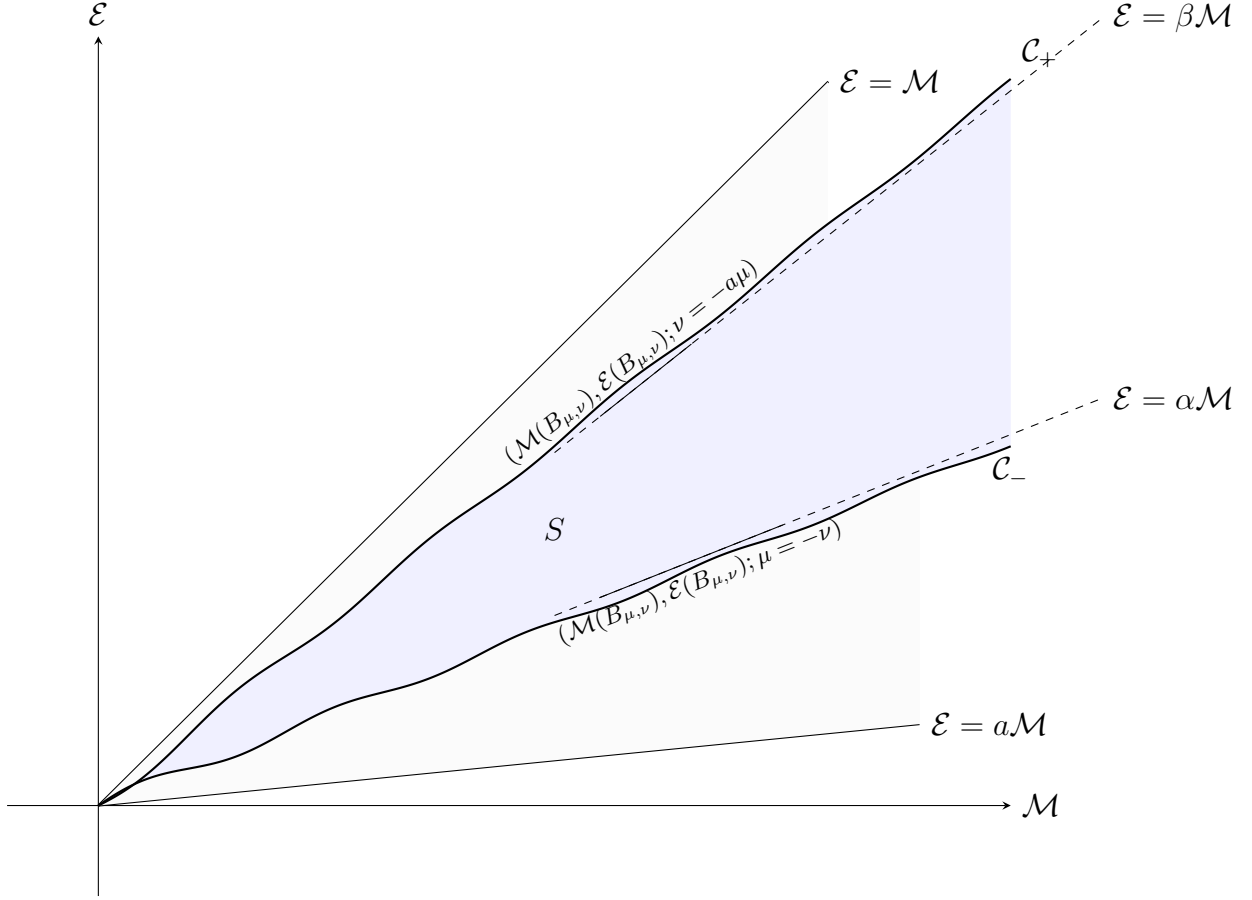

\begin{proof}
Let a measure $\lambda$ so that $$\mathcal{M}(\lambda) =  \mathcal{M}_0,\qquad \mathcal{E}(\lambda)=\mathcal{E}_0.$$
We remind that the entropy for a general measure $\lambda$ decomposed as before as $\dd \lambda = f \dd p + \dd \lambda_{sing}$ is 
$$\mathcal{H}_{qu}(\lambda)  = \int_{ \mathbb{T}^d }[ (1+f)\ln(1+f)  - f \ln(f)] \dd p, $$
and this definition is motivated in the introduction. 

We compute 
\begin{align*}
    & \mathcal{H}_{qu}(\lambda) - \mathcal{H}_{qu}(B_{\mu,\nu}) \\ &= 
    \int 
    [(1+f) \ln( 1+f)  - f\ln(f) ] \ \dd p - \int [ (1+B_{\mu,\nu}) \ln( 1+B_{\mu,\nu})  - B_{\mu,\nu} \ln( B_{\mu,\nu}) ] \dd p\\ 
    &  = \int (1+f) \ln\left( \frac{1+f}{1+B_{\mu,\nu}}\right) + \int f \ln\left( \frac{1+B_{\mu,\nu}}{f} \right) -\int (1+B_{\mu,\nu})\ln (1+B_{\mu,\nu})\\
    & \hspace{8cm} + \int [ B_{\mu,\nu} \ln (B_{\mu,\nu}) + \ln (1+B_{\mu,\nu}) ] 
    \\
    & = \int 
    \left\{ (1+f) \ln\left( \frac{1+f}{1+B_{\mu,\nu}}\right) - f \ln \left( \frac{f}{B_{\mu,\nu}} \right) - f \ln e^{-\mu \omega(p) -\nu} + B_{\mu,\nu} \ln \left(  \frac{B_{\mu,\nu}}{1+B_{\mu,\nu}}\right) \right\} . 
\end{align*}
Now since $1+B_{\mu,\nu} = B_{\mu,\nu} e^{\mu \omega(p) + \nu}$, the right-hand side becomes
\begin{align*}
     & \mathcal{H}_{qu}(\lambda) - \mathcal{H}_{qu}(B_{\mu,\nu}) \\ &= 
\int \left\{ (1+f) \ln\left( \frac{1+f}{1+B_{\mu,\nu}}\right) - f \ln \left( \frac{f}{B_{\mu,\nu}} \right) \right\} + \int (\mu\omega(p) + \nu) (f- B_{\mu, \nu}).
\end{align*}

Now we examine the function 
$\varphi_y(x) = (1+x) \ln\left( \frac{1+x}{1+y}\right)  - x \ln\left( \frac{x}{y}\right) \leq 0$ with strict inequality for $x\neq y$. We compute $\varphi_y' = \ln\left(  \frac{1+x}{1+y}\frac{y}{x}\right)$, which is zero if and only if $x=y$. This implies that 
$$\mathcal{H}_{qu}(\lambda)\leq \mathcal{H}_{qu}(B_{\mu,\nu}) +  \int (\mu\omega(p) + \nu) (f\dd p + \dd \lambda_{sing}) - \mu \mathcal{E} ( B_{\mu, \nu}) - \nu \mathcal{M} ( B_{\mu, \nu}). $$
Or in other words, by the mass and energy constraints on $\lambda$, 
\begin{equation} \label{Quantum entropy ineq} \mathcal{H}_{qu}(\lambda)\leq \mathcal{H}_{qu}(B_{\mu,\nu}) +\mu (\mathcal{E}_0 - \mathcal{E} ( B_{\mu, \nu})) + \nu (\mathcal{M}_0 - \mathcal{M} ( B_{\mu, \nu})).
\end{equation}

Now analogously with the classical case we have the following: 

\medskip
\noindent
(i) If $\mathcal{M}_0, \mathcal{E}_0$ are so that $f_- (\mathcal{M}_0) < \mathcal{E}_0 < f_+ (\mathcal{M}_0) $ (for large masses $\mathcal{M}_0$, we recover the classical case $\mathcal{E}_0/\mathcal{M}_0 \in (\alpha, \beta)$), thanks to Theorem \ref{theorem1_Quantum}, we can choose the parameters $\mu, \nu $ so that $\mathcal{M}(B_{\mu,\nu}) = \mathcal{M}_0$ and $\mathcal{E}(B_{\mu,\nu}) = \mathcal{E}_0$. Then inserting this in the inequality \eqref{Quantum entropy ineq}, we immediately get that 
$\mathcal{H}_{qu}(\lambda)\leq \mathcal{H}_{qu}(B_{\mu,\nu})$ with equality if and only if $\lambda =B_{\mu,\nu}$.

\medskip
\noindent
(ii) If $\mathcal{M}_0, \mathcal{E}_0$ are so that $f_+ (\mathcal{M}_0) \leq \mathcal{E}_0$, thanks to Theorem \ref{theorem1_Quantum} we can choose $\mu,\nu$ so that $\mu=-\nu $, $\nu>0$ and 
$$\mathcal{E}(B_{\mu, \nu}) = f_+ (\mathcal{M}(B_{\mu, \nu})) \ \text{ and }\  \mathcal{E}_0 -  \mathcal{E}(B_{\mu,\nu}) = \mathcal{M}_0 - \mathcal{M}(B_{\mu,\nu})$$
(the existence of a solution $(\mu,\nu)$ to the above can be read off from Figure \ref{figureEM}).

Then inequality \eqref{Quantum entropy ineq} becomes $\mathcal{H}_{qu}(\lambda)\leq \mathcal{H}_{qu}(B_{\mu,\nu})$ with equality when $f= B_{\mu,\nu}$. Since $\mathcal{M}(\lambda_{sing}) = \mathcal{E}(\lambda_{sing})$, we should have $\int \dd \lambda_{sing} = \int \omega \dd \lambda_{sing}$, meaning that the measure $\lambda_{sing}$ is concentrated where $\omega=1$. 

\medskip
\noindent
(iii) Same argument holds in the remaining regime. 
\end{proof}

\bibliographystyle{alpha}
\bibliography{references.bib}

\begin{thebibliography}{OLDC23}

\bibitem[ALS06]{AokiLukkSpohn}
K.~Aoki, J.~Lukkarinen, and H.~Spohn.
\newblock Energy transport in weakly anharmonic chains.
\newblock {\em J Stat Phys}, 124:1105–1129, 2006.

\bibitem[CRZ05]{fiftyFPU}
D.~Campbell, P.~Rosenau, and G.~Zaslavsky.
\newblock Introduction: The fermi–pasta–ulam problem—the first fifty
  years.
\newblock {\em Chaos: An Interdisciplinary Journal of Nonlinear Science}, 15,
  03 2005.

\bibitem[Dha08]{Dhar08}
A.~Dhar.
\newblock Heat transport in low-dimensional systems.
\newblock {\em Advances in Physics}, 57, 08 2008.

\bibitem[DT84]{DemengelTemam}
F.~Demengel and R.~Temam.
\newblock Convex functions of a measure and applications.
\newblock {\em Indiana University Mathematics Journal}, 33:673--709, 1984.

\bibitem[Ein25]{Einstein}
A.~Einstein.
\newblock Quantentheorie des einatomingen idealen gases.
\newblock {\em Sitzungsberichte}, 1:3–14, 1925.

\bibitem[EMV05]{EscMiscValle}
M.~Escobedo, S.~Mischler, and M.~A. Valle.
\newblock Entropy maximisation problem for quantum relativistic particles.
\newblock {\em Bulletin de la Société Mathématique de France}, 133:87--120,
  2005.

\bibitem[EV15]{Esc2015}
M.~Escobedo and J.~J.~L. Vel\'azquez.
\newblock On the theory of weak turbulence for the nonlinear {S}chr\"odinger
  equation.
\newblock {\em Mem. Amer. Math. Soc.}, 238(1124):v+107, 2015.

\bibitem[Gal07]{gallavotti2007fermi}
G.~Gallavotti.
\newblock {\em The Fermi-Pasta-Ulam problem: a status report}, volume 728.
\newblock Springer Science \& Business Media, 2007.

\bibitem[Lep16]{Lep16}
S.~Lepri.
\newblock {\em Thermal Transport in Low Dimensions: From Statistical Physics to
  Nanoscale Heat Transfer}, volume 921.
\newblock 01 2016.

\bibitem[LLP03]{LLP03}
S.~Lepri, R.~Livi, and A.~Politi.
\newblock Thermal conduction in classical low-dimensional lattices.
\newblock {\em Phys. Rep.}, 377(1):1--80, 2003.

\bibitem[LS08]{LukkarinenSpohn2008}
J.~Lukkarinen and H.~Spohn.
\newblock Anomalous energy transport in the {F}{P}{U}-$\beta$ chain.
\newblock {\em Communications on Pure and Applied Mathematics: A Journal Issued
  by the Courant Institute of Mathematical Sciences}, 61(12):1753--1786, 2008.

\bibitem[Luk16]{Luk2016}
J.~Lukkarinen.
\newblock {\em Kinetic Theory of Phonons in Weakly Anharmonic Particle Chains},
  pages 159--214.
\newblock Springer International Publishing, Cham, 2016.

\bibitem[OLDC23]{ONORATO2023}
M.~Onorato, Y.V. Lvov, G.~Dematteis, and S.~Chibbaro.
\newblock Wave turbulence and thermalization in one-dimensional chains.
\newblock {\em Physics Reports}, 1040:1--36, 2023.
\newblock Wave Turbulence and thermalization in one-dimensional chains.

\bibitem[Spo06a]{Spohn2006}
H.~Spohn.
\newblock Collisional invariants for the phonon boltzmann equation.
\newblock {\em Journal of Statistical Physics}, 124:1131–1135, 2006.

\bibitem[Spo06b]{Spohn2005}
H.~Spohn.
\newblock The phonon boltzmann equation, properties and link to weakly
  anharmonic lattice dynamics.
\newblock {\em Journal of Statistical Physics}, 124:1041--1104, 2006.

\end{thebibliography}

\end{document}